\documentclass[a4paper,11pt,pdf]{amsart}
\usepackage{enumerate, amsmath, amsfonts, amssymb, amsthm, thmtools, wasysym, graphics, graphicx, xcolor, frcursive,comment,bbm}

\usepackage{etex}
\usepackage{lscape}
\usepackage{tikz-cd}

\makeatletter
\newcommand{\thickhline}{%
	\noalign {\ifnum 0=`}\fi \hrule height 1pt
	\futurelet \reserved@a \@xhline
}

\definecolor{darkblue}{rgb}{0.0,0,0.7} 

\definecolor{darkred}{rgb}{0.7,0,0} 
\usepackage{hyperref}
\usepackage[all]{xy}
\usepackage[T1]{fontenc}

\usepackage{vmargin}           
\setmarginsrb{2.8cm}{2.5cm}{2.8cm}{2.5cm}{0cm}{0.6cm}{0cm}{0cm}

\usepackage{caption,lipsum}
\usepackage{graphicx}                  
\usepackage{pstricks,pst-plot,pst-text,pst-tree,pst-eps,pst-fill,pst-node,pst-math}
\usepackage{setspace}
\usepackage{multicol}

\newcommand{\darkred}{\color{darkred}} 
\newcommand{\defn}[1]{\emph{\darkred #1}}

\usepackage{etex}

\newtheorem{theorem}{Theorem}[section]
\newtheorem{prop}[theorem]{Proposition}
\newtheorem{lemma}[theorem]{Lemma}
\newtheorem{cor}[theorem]{Corollary}

\theoremstyle{definition}

\newtheorem{rmq}[theorem]{Remark}

\newtheorem{question}[theorem]{Question}

\numberwithin{equation}{section}

\author{Thomas Gobet}
\address{Institut Denis Poisson, CNRS UMR 7013, Faculté des Sciences et Techniques, Université de Tours, Parc de Grandmont, 
37200 TOURS, France}

\title[On maximal dihedral reflection subgroups and noncrossing partitions]{On maximal dihedral reflection subgroups and generalized noncrossing partitions}

\begin{document}
	
\maketitle	
	
\begin{abstract}
In this note, we give a new proof of a result of Matthew Dyer stating that in an arbitrary Coxeter group $W$, every pair $t,t'$ of distinct reflections lie in a unique maximal dihedral reflection subgroup of $W$. Our proof only relies on the combinatorics of words, in particular we do not use root systems at all. As an application, we deduce a new proof of a recent result of Delucchi-Paolini-Salvetti, stating that the poset $[1,c]_T$ of generalized noncrossing partitions in any Coxeter group of rank $3$ is a lattice. We achieve this by showing the more general statement that any interval of length $3$ in the absolute order on an arbitrary Coxeter group is a lattice. This implies that the interval group attached to any interval $[1,w]_T$ where $w$ is an element of an arbitrary Coxeter group with $\ell_T(w)=3$ is a quasi-Garside group.   
\end{abstract}	

\vspace{1cm}

\section{Maximal dihedral reflection subgroups of Coxeter groups}\label{sec_1}

Given a Coxeter system $(W,S)$, we denote by $T$ its set $\bigcup_{w\in W} w S w^{-1}$ of reflections, and by $\ell: W \rightarrow \mathbb{Z}_{\geq 0}$ the length function on $W$ with respect to the generating set $S$. We de not necessarily assume $S$ to be finite. 

We recall the following result of Dyer, which is a consequence of the so-called exchange condition. 

\begin{lemma}\label{palindromic}
Let $(W,S)$ be a Coxeter system. Let $t\in T$ be a reflection of $(W,S)$ and let $t_1 t_2 \cdots t_{2n+1}$ be a reduced expression of $t$ ($t_i\in S$). Then $t=t_1 t_2\cdots t_{n+1} t_n \cdots t_2 t_1$ is a reduced expression of $t$.
\end{lemma}
\begin{proof}
See~\cite[Lemma 1.4]{Dyer_thesis}. 
\end{proof}

We deduce the following.

\begin{lemma}\label{lem_tech}
Let $(W,S)$ be a Coxeter system. Let $w,w'\in W$, $s, t\in S$ such that $w=sw't$, $\ell(w)=\ell(w')+2$, and both $sw', w't$ are reflections of $(W,S)$. Then $s\neq t$ and there is $n\geq 1$ such that $(st)^n$ is a reduced expression of $w$. 
\end{lemma}

\begin{proof}
We argue by induction on $\ell(w')$. If $\ell(w')=0$, then $w=st$ and $\ell(w)=2$, which forces $s\neq t$. We thus get the result with $n=1$. 

Assume that $\ell(w')>0$. Since $\ell(w')=\ell(w)-2$, we have $\ell(w')=\ell(swt)=\ell(sw)-1$. Hence $t$ is a (left or right) descent of the reflection $sw=w't$. Let $t t_1 \cdots t_k t_{k-1} \cdots t_1 t$ be a palindromic reduced expression of $sw$ beginning by $t$. Then $st t_1 \cdots t_k t_{k-1} \cdots t_1$ is a reduced expression of $wt$ since $\ell(sw)=\ell(wt)=\ell(w)-1$). Since by assumption $wt=sw'$ is a reflection, applying Lemma~\ref{palindromic}, we get that $s t t_1 \cdots t_{k-1} t_{k-2} \cdots t_1 t s$ is a palindromic reduced expression of $wt$. Hence $t t_1 \cdots t_{k-1} t_{k-2} \cdots t_1 t s$ is a reduced expression of $swt=w'$. Setting $w''= t_1 \cdots t_{k-1} t_{k-2} \cdots t_1 t$, we have \begin{itemize}
	\item $\ell(w')=\ell(w'')+2$,
	\item $w'= t w'' s$, 
	\item $tw''\in T$ since $tw''=w's=s(sw')s$ and $sw'\in T$, and $w''s\in T$ since $w''s=t w'=t(w't)t$ and $w't\in T$.
\end{itemize}
By induction, we thus get that $t\neq s$ and that there is $n\geq 1$ such that $(ts)^n$ is a reduced expression of $w'$. Since $w=sw't$ and $\ell(w)=\ell(w')+2$, we get that $(st)^{n+1}$ is a reduced expression of $w$. 
\end{proof}

The following two propositions were suggested to me by Jean-Yves Hée, and the proof of Proposition 1.4 is his proof.  

\begin{prop}\label{prop_1}
Let $(W,S)$ be a Coxeter system. Suppose that there is $w\in W$, $w\neq 1$ such that, for every $s\in S$, there is $s'\in T$ such that $w=ss'$. Then $|S|=2$. 
\end{prop}

\begin{proof}
Let $s\in S$ such that $\ell(sw)=\ell(w)-1$ (such an $s$ must exist since $w\neq 1$). Let $t_1 t_2\cdots t_{n+1} t_n \cdots t_1$ be a palindromic reduced expression of $sw\in T$. We claim that $w t_1$ is a reflection. Indeed, by assumption there is $t'\in T$ such that $w=t_1 t'$, hence $w t_1= t_1 t' t_1\in T$. Setting $t:=t_1$, $w':=swt$, we have $\ell(w)=\ell(w')+2$, $sw'=wt\in T$, $w't=sw\in T$. We can thus apply Lemma~\ref{lem_tech} to deduce that $t\neq s$ and that there exists $n\geq 1$ such that $w=(st)^n$.

We claim that $S=\{s,t\}$. Suppose that there is $r\in S\backslash\{ s, t\}$. We thus have $rw\in T$. Since $(s t)^n$ is a reduced expression of $w$ and $r\neq s, t$, the expression $r(s t)^{n}$ of the reflection $rw$ is reduced. Applying Lemma~\ref{palindromic} from the right yields $rw=(t s)^{n} t=t (st)^n$, hence $r=t$ since we also have $rw=r(st)^n$, a contradiction. This shows that $S=\{s, t\}$, which concludes the proof.    
\end{proof}

\begin{prop}\label{prop_2}
Let $(W,S)$ be a Coxeter system. Let $w$ in $W$ be a product of two distinct reflections of $W$. Consider the set $$R_w:=\{t\in T \ \vert \ \text{There exists~}t'\in T\text{~such that~}w=tt'\}.$$ Then setting $W_w:=\langle R_w \rangle$, we have $w\in W_w$, and there exists a subset $S'\subseteq R_w$ such that $(W_w, S')$ is a Coxeter system with set of reflections $R_w$. Moreover, for any subset $S'\subseteq R_w$ such that $(W_w, S')$ is a Coxeter system, we have $|S'|=2$. 
\end{prop}

\begin{proof}
Since $w$ is a product of two distinct reflections, we have $R_w\neq\emptyset$. Assume that $w=tt'$, with $t, t'\in T$. Then $t\in R_w$ but since $w=tt'=t'(t'tt')$ and $t'tt'\in T$, we also have $t'\in R_w$. It implies that $w\in W_w$. 

Now let $t, r\in R_w$. We claim that $rtr$ also lies in $R_w$. By definition of $R_w$, we can find $t',r'\in T$ such that $w=tt'=rr'$. We then have $$w=rr'=r(r'r)r=rw^{-1}r=rt'tr=(rtr)(rtt'tr).$$ Since $rtr, rtt'tr\in T$, we deduce that $rtr\in R_w$.

By a result proven independently by Dyer~\cite{Dyer_subgroups} and Deodhar~\cite{Deodhar} (see also Hée~\cite{Hee}), the reflection subgroup $\langle R_w \rangle$ of $W$ generated by $R_w$ is a Coxeter group $W_w$, and we denote by $S_w \subseteq T$ its canonical simple system. Since, as observed above, the set $R_w$ is stable by conjugation by elements of $R_w$, by~\cite[Corollary 3.11]{Dyer_subgroups} we obtain that $R_w$ is the set of reflections of $W_w$ for this canonical Coxeter group structure, and hence that $S_w\subseteq R_w$.

To conclude the proof, it suffices to apply Proposition~\ref{prop_1} to the Coxeter system $(W_w, S')$ and the element $w$, where $S'\subseteq R_w$ is any subset $S'\subseteq R_w$ such that $(W_w, S')$ is a Coxeter system.    
\end{proof}

One may deduce the following well-known result. 

\begin{cor}\label{cor_sub_dih}
Let $(W,S)$ be a dihedral Coxeter system (i.e., with $|S|=2$) and $W'\subseteq W$ be a reflection subgroup of $W$. Then $W'$ has rank at most two.	
\end{cor}

\begin{proof}
Recall from~\cite{Dyer_subgroups} that the set of reflections of $W'$ is given by $T':=W'\cap T$. Suppose that $W'$ has rank at least two, that is, that there are at least two distinct reflections $r,r'$ in $W'$. Let $w:=rr'\in W'$. By Proposition~\ref{prop_2}, the set $$R_{w}':=\{t\in T' \ \vert \ \text{There exists~}t'\in T'\text{~such that~}w=tt'\}$$ is the set of reflections in a Coxeter system of rank two. We claim that $R_{w}'=T'$. Hence let $t\in T'$. Since $w$ is a product of two reflections of $W$, which is dihedral, and $t\in T$, we have $tw\in T$. As both $t$ and $w$ lie in $W'$ we also have $tw\in W'$. Hence $tw\in T'$, from what we deduce that $t\in R_w'$, yielding the inclusion $T'\subseteq R_w'$. The other inclusion is immediate, hence $R_w'=T'$. 
\end{proof}

The following theorem was proven by Dyer in~\cite[Corollary 3.18]{Dyer_thesis} (see also~\cite[Remark 3.2]{Dyer_Bruhat} and~\cite{Dyer_rank}) using root systems. 

\begin{theorem}\label{thm_max_dih}
Let $t,t'$ be two distinct reflections in an arbitrary Coxeter system $(W,S)$. Then there is a unique maximal dihedral reflection subgroup $W(t,t')$ of $W$ containing $t$ and $t'$. It is given by the dihedral reflection subgroup $W_w$ of Proposition~\ref{prop_2}, where $w=tt'$. 
\end{theorem}

\begin{proof}
Set $w:=tt'$ and consider the dihedral reflection subgroup $W_w$ from Proposition~\ref{prop_2}. Since $tt'=t'(t'tt')$ we have that $t, t'\in W_w$. We claim that $W(t,t'):=W_w$ is the unique maximal dihedral reflection subgroup of $W$ containing $t$ and $t'$. 

Assume that $W'$ is a dihedral reflection subgroup of $W$ containing both $t$ and $t'$. Then $W'$ contains $w$. Since $w$ is a product of two reflections and $W'$ is dihedral, then for any reflection $r$ of $W'$ we have that $rw$ is a reflection of $W'$ (hence of $W$). Hence $r\in R_w \subseteq W_w$, and thus $W_w$ contains every reflection of $W'$. It follows that $W'\subseteq W_w$, which concludes the proof.  
\end{proof}

We list a few properties of maximal dihedral reflection subgroups which we shall use in the next section (see also~\cite[Lemma 3.1]{Dyer_Bruhat} for point~(2)). 

\begin{lemma}\label{lem_4}
Let $t_1, t_2, t_3, t_4$ be reflections in $T$ such that $t_1\neq t_2$ and $t_1t_2=t_3t_4$. Then \begin{enumerate}
\item $t_3\neq t_4$ and $W(t_1,t_2)=W(t_3, t_4)$,
\item The reflection subgroup $W'=\langle t_1, t_2, t_3, t_4 \rangle$ is dihedral. 
\item If $1\neq w=rr' \in W(t_1, t_2)$ with $r,r'\in T$, then $W(r,r')= W(t_1, t_2)$ and hence $r,r'\in W(t_1, t_2)$. 
\end{enumerate}
\end{lemma}

\begin{proof}
Thanks to Theorem~\ref{thm_max_dih}, we have $W(t_1, t_2)=W_{t_1t_2}=W_{t_3t_4}=W(t_3, t_4)$, which yields $(1)$. 	
	
For $(2)$, by $(1)$ we have that $t_i\in W(t_1, t_2)$ for all $i=1, 2, 3, 4$, hence $W'$ is a reflection subgroup of $W(t_1,t_2)$. By Corollary~\ref{cor_sub_dih} above and since $t_1 \neq t_2$, it is dihedral. 

For $(3)$, let $q, q'\in W(t_1, t_2)\cap T$ such that $w=q q'$. Then $W(q,q')$ and $W(t_1, t_2)$ are two maximal dihedral reflection subgroups containing $q$ and $q'$, we thus have $W(t_1, t_2)=W(q, q')$ by Theorem~\ref{thm_max_dih}. Now since $qq'=rr'$, by $(1)$ we conclude that $W(q,q')=W(r,r')$. 
\end{proof}

\section{Application: lattice property of intervals of length three in noncrossing partitions}

Let $(W,S)$ be a Coxeter system. Denote by $\ell_T: W \longrightarrow \mathbb{Z}_{\geq 0}$ the length function on $W$ with respect to the generating set $T:=\bigcup_{w\in W} w S w^{-1}$ of $W$, and recall that we denote by $\ell$ the length function with respect to $S$. 

The \defn{absolute order} $\leq_T$ on $W$ is defined by $u\leq_T v$ if and only if $$\ell_T(u)+\ell_T(u^{-1}v)=\ell_T(v).$$ Given $u, v\in W$ with $u\leq_T v$, we denote by $[u,v]_T$ the corresponding interval in the absolute order, that is, the partially ordered set $\{w\in W~|~u\leq_T w\leq_T v\}$ endowed with the restriction of $\leq_T$. 

If $S$ is finite, a \defn{(standard) Coxeter element} $c$ in $W$ is the product of all the elements of $S$ in some order. One has $\ell(c)=\ell_T(c)=n:=|S|$ (this can be deduced for instance from~\cite[Theorem 1.1]{Dyer_refn}). If $c$ is a Coxeter element in $W$, the poset $[1,c]_T$ is sometimes called the poset of \defn{(generalized) noncrossing partitions} (associated to the Coxeter group $W$ and to the Coxeter element $c$). The reason for such a terminology is that in type $A_n$, the poset $[1,c]_T$ is isomorphic to the poset of combinatorial noncrossing partitions of the set $\{1,2,\dots, n+1\}$ (ordered by refinement)--see~\cite[Theorem 1]{Biane}.

Note that, since $T$ is stable by conjugation, if $c$ and $c'$ are Coxeter elements which are conjugate in $W$, say $c'=w c w^{-1}$ for some $w\in W$, then the posets $[1,c]_T$ and $[1,c']_T$ are isomorphic via the map $x \mapsto w x w^{-1}$. In a Coxeter system there can be several conjugacy classes of Coxeter elements in general, but in type $A_n$ and more generally in any Coxeter system whose underlying Coxeter graph is a tree, all the Coxeter elements are (cyclically) conjugate~\cite[Theorem 3.1.4]{GP}. In type $A_n$, the interval $[1,c]_T$ is a lattice since combinatorial noncrossing partitions form a lattice. More generally, the poset $[1,c]_T$ is known to be a lattice for every Coxeter element $c$ in every finite Coxeter group~\cite{Dual, BW, Reading}, in types $\widetilde{A}_n$ and $\widetilde{C}_n$ for some choices of Coxeter elements~\cite{Dig1, Dig2}, in type $\widetilde{G}_2$~\cite{McCammond}, for universal Coxeter systems~\cite{Bessis_free}, and for all Coxeter systems of rank three~\cite{DPS}. In those affine types which are not mentioned above, it has been shown that $[1,c]_T$ is never a lattice~\cite{McCammond}. For other Coxeter groups, to the best of our knowledge this question is completely open.  

The lattice property of $[1,c]_T$ has important applications in the study of the Artin-Tits group $B_W$ attached to $W$.  The Artin-Tits group $B_W=B(W,S)$ attached to a Coxeter system $(W,S)$ is the group with the same presentation as $(W,S)$, except that the relations $s^2=1$, $\forall s\in S$ are removed. There is a canonical surjection $B_W\twoheadrightarrow W$ induced by $\mathbf{s}\mapsto s$, where $\mathbf{s}$ is the Artin--Tits group generator corresponding to $s\in S$. Unlike in Coxeter groups, the word problem in Artin--Tits groups is still open in general. When $[1,c]_T$ is a lattice, the so-called interval group $G([1,c]_T)$ which can be built from this lattice is a (quasi-)Garside group (see~\cite[Section 2.2]{Neaime} and the references therein for basics on interval Garside structures). For $W$ of spherical type, this group is known to be isomorphic to $B_W$~\cite{Dual}, and this yields new proofs that the word problem in $B_W$ is solvable, as Garside groups have a solvable word problem (among other remarkable properties). See~\cite{Dual, Garside} for more on the topic. 

It is thus natural to try to address the following question:

\begin{question}
	For which Coxeter systems $(W,S)$ and which choice of Coxeter element $c$ is the poset $[1,c]_T$ a lattice ? 
\end{question}

As an application of the results of Section~\ref{sec_1}, we give a new proof of a recent theorem of Delucchi, Paolini and Salvetti~\cite[Theorem 4.2]{DPS}, stating that $[1,c]_T$ is a lattice in all Coxeter groups of rank three, which only relies on the combinatorics of words, by showing the following more general result.  

\begin{theorem}[{Lattice property of length three intervals in absolute order}]\label{thm:main}
	Let $(W,S)$ be a Coxeter system. Let $u, v\in W$ such that $u\leq_T v$ and $\ell_T(v)=\ell_T(u)+3$. Then the interval $[u,v]_T$ is a lattice. In particular, for every Coxeter element $c$ in a Coxeter system $(W,S)$ of rank three, the poset $[1,c]_T$ of generalized noncrossing partitions is a lattice. 
\end{theorem}

\begin{proof}
The interval $[u,v]_T$ is isomorphic to $[1, u^{-1} v]_T$. We can thus assume that $u=1$, and $\ell_T(v)=3$. To conclude the proof, it suffices to show the following: let $t\neq t'\in T$, $w,w'\in [1,v]_T$ with $\ell_T(w)=\ell_T(w')=2$, $t, t'\leq_T w$, and $t, t'\leq_T w'$. Then $w=w'$. Indeed, since $\ell_T(v)=3$, any obstruction to the lattice property must come from a so-called "bowtie"~\cite[Definition 1.5]{McCammond}, that is in our setting, a pair $(t_1, t_2)$ of distinct elements of reflection length $1$ together with a pair $(w_1, w_2)$ of distinct elements of reflection length $2$, such that $t_i \leq_T w_j$, for all $i$ and $j$. 

Since $t, t'\leq_T w$, there are $t_1, t_2\in T$ such that $tt_1=t't_2=w$. The subgroup $W_w$ is the maximal dihedral reflection subgroup containing $t$ and $t_1$, and by Lemma~\ref{lem_4} we also have $t', t_2\in W_w$. Hence $W_w=W(t,t')$ since $W_w$ is a maximal dihedral reflection subgroup and contains both $t$ and $t'$. In particular we have $w\in W(t,t')$. Similarly, we have $w'\in W(t,t')$. 

 Now since $w,w'\leq_T v$, there are $r_1, r_2\in T$ such that $wr_1=w'r_2=v$. It implies that $w'^{-1}w=r_2 r_1\in W(t,t')$ since both $w$ and $w'$ lie in $W(t,t')$. If $w\neq w'$, then $r_2\neq r_1$, hence $\ell_T(w'^{-1}w)=2$. By Lemma~\ref{lem_4}~(3) we deduce that $r_1, r_2\in W(t,t')$.  
We thus have that $v=wr_1\in W(t,t')$, a contradiction: since $W(t,t')$ is a dihedral reflection subgroup of $W$, every $x\in W(t,t')$ can be written as a product of at most two reflections of $W(t,t')$ (hence of $W$), hence satisfies $\ell_T(x)\leq 2$, while $\ell_T(v)=3$. Thus $w=w'$, which concludes the proof.  
\end{proof}

We deduce the following corollary. 

\begin{cor}
Let $(W,S)$ be a Coxeter system and $w\in W$ such that $\ell_T(w)=3$. Then the interval group $G([1,w]_T)$ is a quasi-Garside group. 
\end{cor}

\begin{proof} This is an immediate consequence of the lattice property of $[1,w]_T$ and of~\cite[Theorem 5.4]{Dig1} (see also~\cite[Theorem 2.8]{Neaime}). 
\end{proof}

\begin{rmq}
For some particular finite Coxeter groups $W$ (like $H_3$) and some particular $w\in W$ with $\ell_T(w)=3$ called (proper) quasi-Coxeter elements, the lattice property of $[1,w]_T$ was established by computer on a case-by-case basis in~\cite{BN2}, where the corresponding interval groups are also studied.  
\end{rmq}

\textbf{Acknowledgements}. I am extremely grateful to Jean-Yves Hée, who suggested Propositions 1.3 and 1.4 to me. The proof of Proposition 1.4 is his proof, while his original proof of Proposition 1.3 was using root systems. I also thank him for his careful reading of a preliminary version of this article and several insightful comments and suggestions which led to improvements in the exposition, especially for his suggestion to present my proof of Proposition 1.3 using Lemma 1.2.

\end{document}